\documentclass[12pt,reqno]{amsart}
\usepackage{etoolbox}

\makeatletter
\patchcmd\maketitle
  {\uppercasenonmath\shorttitle}
  {}
  {}{}
\patchcmd\maketitle
  {\@nx\MakeUppercase{\the\toks@}}
  {\the\toks@}
  {}
  {}{}
\patchcmd\@settitle{\uppercasenonmath\@title}{\Large}{}{}
\patchcmd\@setauthors
  {\MakeUppercase{\authors}}
  {\authors}
  {}{}
\makeatother
\usepackage{amsmath,amssymb,amsthm,mathdots}
\usepackage{color}
\usepackage{url}
\usepackage{tikz-cd}
\usepackage[utf8]{inputenc}
\usepackage[T1]{fontenc}
\newtheorem{theorem}{Theorem}[section]

\newtheorem{lemma}{Lemma}[section]
\newtheorem{remark}{Remark}[section]

\numberwithin{equation}{section}
\usepackage[colorlinks=true]{hyperref}
\hypersetup{urlcolor=blue, citecolor=red , linkcolor= blue}
\usepackage[capitalise,noabbrev,nameinlink]{cleveref}      
\begin{document}
\address{$^{[1]}$ University of Sfax, Sfax, Tunisia.}
\email{\url{kais.feki@hotmail.com}}


\date{\today}
\author[Kais Feki] {\Large{Kais Feki}$^{1}$}
\title[Some $\mathbb{A}$-numerical radius inequalities for $d\times d$ operator matrices]{Some $\mathbb{A}$-numerical radius inequalities for $d\times d$ operator matrices}

\maketitle
\begin{abstract}
Let $A$ be a positive (semidefinite) bounded linear operator acting on a complex Hilbert space $\big(\mathcal{H}, \langle \cdot\mid \cdot\rangle \big)$. The semi-inner product ${\langle x\mid y\rangle}_A := \langle Ax\mid y\rangle$, $x, y\in\mathcal{H}$
induces a seminorm ${\|\cdot\|}_A$ on $\mathcal{H}$. Let $T$ be an $A$-bounded operator on $\mathcal{H}$, the $A$-numerical radius of $T$ is given by
\begin{align*}
\omega_A(T) = \sup\Big\{\big|{\langle Tx\mid x\rangle}_A\big|: \,\,x\in \mathcal{H}, \,{\|x\|}_A = 1\Big\}.
\end{align*}
In this paper, we establish several inequalities for $\omega_\mathbb{A}(\mathbb{T})$, where $\mathbb{T}=(T_{ij})$ is a $d\times d$ operator matrix with $T_{ij}$ are $A$-bounded operators and $\mathbb{A}$ is the diagonal operator matrix whose each diagonal entry is $A$. 
\end{abstract}

\section{Introduction and Preliminaries}\label{s1}
Let $\left(\mathcal{H}, \left\langle \cdot\mid \cdot\right\rangle \right)$ be a non-trivial complex Hilbert space, and let $\mathcal{B}(\mathcal{H})$
denote the $C^*$-algebra of all bounded linear operators on $\mathcal{H}$ with identity $I_{\mathcal{H}}$ (or $I$ if no confusion arises). If $\mathcal{H}=\mathbb{C}^d$, we identify $\mathcal{B}(\mathbb{C}^d)$ with the matrix algebra $\mathbb{M}_d(\mathbb{C})$ of $d\times d$ complex matrices. Let $\mathcal{B}(\mathcal{H})^+$ be the cone of positive (semi-definite) operators, i.e., $\mathcal{B}(\mathcal{H})^+=\left\{A\in \mathcal{B}(\mathcal{H})\,:\,\langle Ax\mid x\rangle\geq 0,\;\forall\;x\in \mathcal{H}\;\right\}$. Every $A\in \mathcal{B}(\mathcal{H})^+$ defines the following positive semi-definite sesquilinear form:
$$\langle\cdot\mid\cdot\rangle_{A}:\mathcal{H}\times \mathcal{H}\longrightarrow\mathbb{C},\;(x,y)\longmapsto\langle x\mid y\rangle_{A} =\langle Ax\mid y\rangle.$$
Clearly, the induced semi-norm is given by $\|x\|_A=\langle x\mid x\rangle_A^{1/2}$, for every $x\in \mathcal{H}$. This makes $\mathcal{H}$ into a semi-Hilbertian space. One can verify that $\|\cdot\|_A$ is a norm on $\mathcal{H}$ if and only if $A$ is injective, and that $(\mathcal{H},\|\cdot\|_A)$ is complete if and only if the range of $A$ is a closed subspace of $\mathcal{H}$.

Throughout this article, we shall assume that an operator $A\in\mathcal{B}(\mathcal{H})$ is a nonzero positive (semidefinite) operator. Moreover, by an operator we mean a bounded linear operator. In addition, the range and the null space of an operator $T$ are denoted by ${\mathcal R}(T)$ and ${\mathcal N}(T)$, respectively. Also, $T^*$ will be denoted to be the adjoint of $T$.

An operator $S\in\mathcal{B}(\mathcal{H})$ is called an $A$-adjoint of $T$ if for every $x,y\in \mathcal{H}$, the identity $\langle Tx\mid y\rangle_A=\langle x\mid Sy\rangle_A$ holds. The existence of an $A$-adjoint operator is not guaranteed. Observe that $T$ admits an $A$-adjoint operator if and only if the equation $AX = T^*A$ has solution. This kind of equations can be studied by using the next theorem due to Douglas (for its proof see \cite{doug} or \cite{mos2019}).
\begin{theorem}\label{doug}(\cite[Theorem 1]{doug})
If $T, S \in \mathcal{B}(\mathcal{H})$, then the following statements are equivalent:
\begin{enumerate}
\item[{\rm (i)}] $\mathcal{R}(S) \subseteq \mathcal{R}(T)$;
\item[{\rm (ii)}] $TD=S$ for some $D\in \mathcal{B}(\mathcal{H})$;
\item[{\rm (iii)}]$SS^* \leq \lambda^2 TT^*$ for some $\lambda\geq 0$ (or equivalently $\|S^*x\| \leq \lambda\|T^*x\|$ for all $x\in \mathcal{H}$).
\end{enumerate}
If one of these conditions holds, then there exists an unique operator $Q\in\mathcal{B}(\mathcal{H})$ such that $TX=S$ and $\mathcal{R}(Q) \subseteq \overline{\mathcal{R}(T^{*})}$. Furthermore, $\mathcal{N}(Q)=\mathcal{N}(S)$ and
$$\|Q\|^2=\inf\left\{\mu\,;\;SS^*\leq \mu TT^*\right\}.$$
Such $Q$ is called the reduced solution or Douglas solution of $TX=S$.
\end{theorem}
Therefore, if we denote by $\mathcal{B}_{A}(\mathcal{H})$ the subalgebra of $\mathcal{B}(\mathcal{H})$ of all operators which admit an $A$-adjoint
operator, then by Theorem \ref{doug} we see that
$$\mathcal{B}_{A}(\mathcal{H})=\left\{T\in \mathcal{B}(\mathcal{H})\,;\;\mathcal{R}(T^{*}A)\subset \mathcal{R}(A)\right\}.$$
Let $T\in\mathcal{B}_{A}(\mathcal{H})$. The Douglas solution of the equation $AX = T^*A$ is a distinguished $A$-adjoint operator of $T$, which is denoted by $T^{\sharp_A}$. Note that, $T^{\sharp_A} = A^{\dag}T^*A$ in which $A^{\dag}$ is denoted to be the Moore-Penrose inverse of $A$ (see \cite{acg2}). It is important to mention that if $T\in\mathcal{B}_{A}(\mathcal{H})$, then $T^{\sharp_A}\in\mathcal{B}_{A}(\mathcal{H})$, $\|T^{\sharp_A}\|_A=\|T\|_A$ and $(T^{\sharp_A})^{\sharp_A} = P_{\overline{\mathcal{R}(A)}}TP_{\overline{\mathcal{R}(A)}}$. Here, $P_{\overline{\mathcal{R}(A)}}$ denotes the orthogonal projection onto $\overline{\mathcal{R}(A)}$. Furthermore, if $T, S\in\mathcal{B}_{A}(\mathcal{H})$, then $(TS)^{\sharp_A} = S^{\sharp_A}T^{\sharp_A}$. In addition, an operator $U\in  \mathcal{B}_A(\mathcal{H})$ is said to be $A$-unitary if $\|U^{\sharp_A}x\|_A= \|Ux\|_A=\|x\|_A$ for all $x\in \mathcal{H}$. For more details, the reader is invited to consult \cite{acg1,acg2,bakfeki01,bakfeki04} and their references.

Furthermore, again by applying Douglas theorem we obtain
\begin{equation}\label{abbbbbbbb}
\mathcal{B}_{A^{1/2}}(\mathcal{H})=\left\{T \in \mathcal{B}(\mathcal{H})\,;\;\exists \,\lambda > 0\,;\;\|Tx\|_{A} \leq \lambda \|x\|_{A},\;\forall\,x\in \mathcal{H}  \right\}.
\end{equation}
Operators in $\mathcal{B}_{A^{1/2}}(\mathcal{H})$ are called $A$-bounded. It should be mention here that $\mathcal{B}_{A}(\mathcal{H})$ and $\mathcal{B}_{A^{1/2}}(\mathcal{H}))$ are two subalgebras of $\mathcal{B}(\mathcal{H})$ which are neither closed nor dense in $\mathcal{B}(\mathcal{H})$. Moreover, we have $\mathcal{B}_{A}(\mathcal{H}) \subseteq \mathcal{B}_{A^{1/2}}(\mathcal{H}))$ (see \cite[Proposition 1.2.]{acg3}).

The semi-inner product $\langle\cdot\mid\cdot\rangle_{A}$ induces the following seminorm on $\mathcal{B}_{A^{1/2}}(\mathcal{H})$:
\begin{equation}\label{semiiineq}
\|T\|_A:=\sup_{\substack{x\in \overline{\mathcal{R}(A)},\\ x\not=0}}\frac{\|Tx\|_A}{\|x\|_A}=\sup\left\{\|Tx\|_{A}\,;\;x\in \mathcal{H},\,\|x\|_{A}= 1\right\}<\infty.
\end{equation}
If $A=I$, we get the classical definition of the operator norm of an operator $T$ which will be denoted by $\|T\|$. It was shown in \cite{fg} that for every $T\in\mathcal{B}_{A^{1/2}}(\mathcal{H})$ we have
\begin{equation}\label{fg}
\|T\|_A=\sup\left\{|\langle Tx\mid y\rangle_A|\,;\;x,y\in \mathcal{H},\,\|x\|_{A}=\|y\|_{A}= 1\right\}.
\end{equation}
In addition, for every $T\in\mathcal{B}_{A}(\mathcal{H})$ we have
\begin{equation}\label{diez}
\|T\|_A^2={\|T^{\sharp_A}T\|}_A = {\|TT^{\sharp_A}\|}_A.
\end{equation}
The $A$-the numerical radius and the $A$-spectral radius of an $A$-bounded operator $T\in\mathcal{B}_{A^{1/2}}(\mathcal{H})$ are defined by
\begin{align*}
\omega_A(T) = \sup\Big\{\big|{\langle Tx\mid x\rangle}_A\big|: \,\,x\in \mathcal{H}, \,{\|x\|}_A = 1\Big\}\;\text{ and }
\end{align*}
\begin{equation}\label{newrad}
r_A(T):=\displaystyle\inf_{n\in \mathbb{N}^*}\|T^n\|_A^{\frac{1}{n}}=\displaystyle\lim_{n\to\infty}\|T^n\|_A^{\frac{1}{n}},
\end{equation}
respectively. Notice that the second equality in \eqref{newrad} is proved in \cite{feki01}. If $A=I$, the spectral and numerical radius of $T$ will be simply denoted by $r(T)$ and $\omega(T)$ respectively.

It is well known that $\omega_A(\cdot)$ defines a seminorm on $\mathcal{B}_{A^{1/2}}(\mathcal{H})$, which is equivalent to the $A$-operator seminorm $ \left\| \cdot \right\|_A $, more precisely,
\begin{equation*}
\tfrac{1}{2} \left\|T\right\|_A \leq \omega_A(T)\leq \left\|T\right\|_A,
\end{equation*}
for every $T\in\mathcal{B}_{A^{1/2}}(\mathcal{H})$. Moreover, it was shown in \cite{feki01} that for $T\in\mathcal{B}_{A^{1/2}}(\mathcal{H})$, it holds
\begin{equation*}
\omega_A(T)\leq \frac{1}{2}\left(\|T\|_A+\|T^2\|_A^{1/2}\right).
\end{equation*}
So, clearly, if $T\in\mathcal{B}_{A^{1/2}}(\mathcal{H})$ and satisfies $AT^2=0$, then
\begin{equation}\label{at2}
\omega_A(T)= \frac{1}{2}\|T\|_A.
\end{equation}

It should be emphasized here that for every $T\in\mathcal{B}_{A^{1/2}}(\mathcal{H})$ we have
\begin{equation}\label{dom}
r_A(T) \leq \omega_A(T).
\end{equation}
Also $r_A(\cdot)$ satisfies the commutativity property, which asserts that
\begin{equation}\label{commut}
r_A(TS)=r_A(ST),
\end{equation}
for every $T,S\in \mathcal{B}_{A^{1/2}}(\mathcal{H})$.

An operator $T\in\mathcal{B}(\mathcal{H})$ is said to be $A$-selfadjoint if $AT$ is selfadjoint, that is, $AT = T^*A$. Moreover, it was shown in \cite{feki01} that if $T$ is $A$-self-adjoint, then
\begin{equation}\label{aself1}
\|T\|_{A}=\omega_A(T)=r_A(T).
\end{equation}

For the sequel, for any arbitrary operator $T\in {\mathcal B}_A({\mathcal H})$, we write
$$\Re_A(T):=\frac{T+T^{\sharp_A}}{2}\;\;\text{ and }\;\;\Im_A(T):=\frac{T-T^{\sharp_A}}{2i}.$$
It has recently been shown in \cite[Theorem 2.5]{zamani1} that if $T\in\mathcal{B}_{A}(\mathcal{H})$, then
\begin{align}\label{zamnum}
\omega_A(T) = \displaystyle{\sup_{\theta \in \mathbb{R}}}{\left\|\Re_A(e^{i\theta}T)\right\|}_A.
\end{align}

Recently, many results covering some classes of operators on a complex Hilbert space $\big(\mathcal{H}, \langle \cdot\mid \cdot\rangle\big)$
are extended to $\big(\mathcal{H}, {\langle \cdot\mid \cdot\rangle}_A\big)$ (see, e.g., \cite{fekisidha2019,feki01,zamani2,bakfeki01,bakfeki04,zamani1,majsecesuci}).

In in this work, we consider the following diagonal operator matrix whose each diagonal entry is $A$:
 \begin{equation*}
\mathbb{A}=\left( {\begin{array}{*{20}{c}}
   {A} & {} & {0} & {}  \\
   {} & {A} &  & {}  \\
   {} &  &  \ddots  & {}  \\
   {} & {0} & {} & {A}  \\
\end{array}} \right),
\end{equation*}
acting on the Hilbert space $\mathbb{H}=\oplus_{i=1}^d\mathcal{H}$ equipped with the following inner-product:
$$\langle x, y\rangle=\sum_{k=1}^d\langle x_k\mid y_k\rangle,$$
 for all $x=(x_1,\cdots,x_d)\in \mathbb{H}$ and $y=(y_1,\cdots,y_d)\in \mathbb{H}$. The semi-inner product induced by the positive operator $\mathbb{A}$ is given by
$$\langle x, y\rangle_{\mathbb{A}}= \langle \mathbb{A}x, y\rangle=\sum_{k=1}^d\langle Ax_k\mid y_k\rangle=\sum_{k=1}^d\langle x_k\mid y_k\rangle_A,$$
 for all $x=(x_1,\cdots,x_d)\in \mathbb{H}$ and $y=(y_1,\cdots,y_d)\in \mathbb{H}$. The purpose of this paper is to establish several inequalities for $\omega_\mathbb{A}(\mathbb{T})$, where $\mathbb{T}=(T_{ij})$ is a $d\times d$ operator matrix with $T_{ij}$ are $A$-bounded operators. The inspiration for our investigation comes from \cite{OK2,S.D.M,BP,bpnayek}.

\section{Results}
In this section, we present our results. To prove our first result, we need the following lemmas.
\begin{lemma}\label{mjom}(\cite{feki03})
Let $\mathbb{T}= (T_{ij})_{d \times d}$ be a $d \times d$ operator matrix be such that $T_{ij}\in \mathcal{B}_{A^{1/2}}(\mathcal{H})$ for all $i,j$. Then, $$r_{\mathbb{A}}(\mathbb{T})\leq r(\|T_{ij} \|_A).$$
\end{lemma}

\begin{lemma}\label{ir2020}(\cite{bhunfekipaul})
Let $\mathbb{T}= (T_{ij})_{d \times d}$ be such that $T_{ij}\in \mathcal{B}_{A}(\mathcal{H})$ for all $i,j$. Then, $\mathbb{T}\in\mathcal{B}_{\mathbb{A}}(\mathbb{H})$ and
$$\mathbb{T}^{\sharp_\mathbb{A}}=(T_{ji}^{\sharp_\mathbb{A}})_{d \times d}.$$
\end{lemma}

\begin{lemma}\label{weak}(\cite{bhunfekipaul})
Let $T\in \mathcal{B}_{A^{1/2}}(\mathcal{H})$. Then,
\begin{equation*}
\omega_A(U^{\sharp_A}TU)=\omega_A(T),
\end{equation*}
for any $A$-unitary operator $U\in\mathcal{B}_A(\mathcal{H})$.
\end{lemma}

Now, we are in a position to prove the following theorem.
\begin{theorem}\label{thf1}
Let $\mathbb{T}=(T_{ij})$ be a ${d\times d}$ operator matrix where $T_{ij}\in \mathcal{B}_{A}(\mathcal{H})$. Then,
\begin{equation*}
\omega_{\mathbb{A}}(\mathbb{T})\leq \frac{1}{2}\sum_{i=1}^d\left(\|T_{ii}\|_A+\sqrt{\left\|T_{ii}T_{ii}^{\sharp_A}+\sum^d_{j=1,j\neq i}T_{ij}T_{ij}^{\sharp_A}\right\|_A}\right).
\end{equation*}
\end{theorem}
\begin{proof}
We first prove that
\begin{equation}\label{first01}
\omega_{\mathbb{A}}(\mathbb{S})\leq \frac{1}{2}\left(\|T_{11}\|_A+\sqrt{\left\|\sum^d_{j=1}T_{1j}T_{1j}^{\sharp_A}\right\|_A}\right),
\end{equation}
where $\mathbb{S}=\begin{pmatrix}
T_{11} & T_{12} &\cdots& T_{1d}\\
0 &0 &\cdots& 0\\
\vdots & \vdots & \vdots & \vdots\\
0 & 0 &\cdots& 0\\
\end{pmatrix}$. Let $\theta\in \mathbb{R}$. It is not difficult to verify that $\Re_A(e^{i\theta}\mathbb{S})$ is an $\mathbb{A}$-self-adjoint operator. So, by \eqref{aself1} we have
\begin{align}\label{rr1}
r_{\mathbb{A}}\left(\Re_\mathbb{A}(e^{i\theta}\mathbb{S})\right)=\|\Re_\mathbb{A}(e^{i\theta}\mathbb{S})\|_{\mathbb{A}}.
\end{align}
On the other hand, by using Lemma \ref{ir2020} we see that
\begin{align*}
r_{\mathbb{A}}\left[\Re_A(e^{i\theta}\mathbb{S})\right]
& =\tfrac{1}{2}r_{\mathbb{A}}(e^{i\theta}\mathbb{S}+e^{-i\theta}\mathbb{S}^{\sharp_{\mathbb{A}}})\\
 &=\frac{1}{2}r_{\mathbb{A}}\left[e^{i\theta}\begin{pmatrix}
T_{11} & T_{12} &\cdots& T_{1d}\\
0 &0 &\cdots& 0\\
\vdots & \vdots & \vdots & \vdots\\
0 & 0 &\cdots& 0\\
\end{pmatrix}+e^{-i\theta}\begin{pmatrix}
T_{11}^{\sharp_{A}} & 0 &\cdots& 0\\
T_{12}^{\sharp_{A}} &0 &\cdots& 0\\
\vdots & \vdots & \vdots & \vdots\\
T_{1d}^{\sharp_{A}} & 0 &\cdots& 0\\
\end{pmatrix}\right]\\
&=\frac{1}{2}r_{\mathbb{A}}\left[\begin{pmatrix}
e^{i\theta}T_{11}+e^{-i\theta}T_{11}^{\sharp_{A}} & e^{i\theta}T_{12} &\cdots& e^{i\theta}T_{1d}\\
e^{-i\theta}T_{12}^{\sharp_{A}} &0 &\cdots& 0\\
\vdots & \vdots & \vdots & \vdots\\
e^{-i\theta}T_{1d}^{\sharp_{A}} & 0 &\cdots& 0\\
\end{pmatrix}\right]\\
&=\frac{1}{2}r_{\mathbb{A}}\left[
\begin{pmatrix}
T_{11}^{\sharp_{A}} & e^{i\theta}I &\cdots& 0\\
T_{12}^{\sharp_{A}} &0 &\cdots& 0\\
\vdots & \vdots & \vdots & \vdots\\
T_{1d}^{\sharp_{A}} & 0 &\cdots& 0\\
\end{pmatrix}
\begin{pmatrix}
e^{-i\theta}I &0 &\cdots& 0\\
T_{11} & T_{12} &\cdots& T_{1d}\\
\vdots & \vdots & \vdots & \vdots\\
0 & 0 &\cdots& 0\\
\end{pmatrix}
\right].
\end{align*}
So, by using \eqref{rr1} together with \eqref{commut} we get
\begin{align*}
\|\Re_A(e^{i\theta}\mathbb{S})\|_{\mathbb{A}}
&=\frac{1}{2}r_{\mathbb{A}}\left[
\begin{pmatrix}
e^{-i\theta}I &0 &\cdots& 0\\
T_{11} & T_{12} &\cdots& T_{1d}\\
\vdots & \vdots & \vdots & \vdots\\
0 & 0 &\cdots& 0\\
\end{pmatrix}
\begin{pmatrix}
T_{11}^{\sharp_{A}} & e^{i\theta}I &\cdots& 0\\
T_{12}^{\sharp_{A}} &0 &\cdots& 0\\
\vdots & \vdots & \vdots & \vdots\\
T_{1d}^{\sharp_{A}} & 0 &\cdots& 0\\
\end{pmatrix}
\right]\\
&=\tfrac{1}{2}r_{\mathbb{A}}\left[
\begin{pmatrix}
e^{-i\theta}T_{11}^{\sharp_{A}} &I &0&\cdots& 0\\
\sum_{k=1}^d T_{1k}T_{1k}^{\sharp_{A}}& e^{i\theta}T_{11} &0&\cdots& 0\\
0 & 0 &0&\cdots& 0\\
\vdots & \vdots & \vdots & \vdots\\
0 & 0 &0&\cdots& 0\\
\end{pmatrix}
\right]\\
&\leq\frac{1}{2}r\left[
\begin{pmatrix}
\|T_{11}\|_A &1 &0&\cdots& 0\\
\left\|\sum_{k=1}^d T_{1k}T_{1k}^{\sharp_{A}}\right\|_A& \|T_{11}\|_A &0&\cdots& 0\\
0 & 0 &0&\cdots& 0\\
\vdots & \vdots & \vdots & \vdots\\
0 & 0 &0&\cdots& 0\\
\end{pmatrix}
\right]\; (\text{by Lemma }\ref{mjom}).
\end{align*}
Hence, we infer that
$$\|\Re_A(e^{i\theta}\mathbb{S})\|_{\mathbb{A}}\leq\frac{1}{2}\left(\|T_{11}\|_A+\sqrt{\left\|\sum^d_{j=1}T_{1j}T_{1j}^{\sharp_A}\right\|_A}\right).$$
So, by taking the supremum over all $\theta\in \mathbb{R}$ in the above inequality and then using \eqref{zamnum} we get \eqref{first01} as desired. Now, for $k\in\{2,\cdots,d\}$, we let
\[
\mathbb{U}_k = \left(\begin{array}{c|c}
    \mathbb{J}_{k\times k}       & 0_{k\times (d-k)}     \\
        \hline
    0_{(d-k)\times k}   & \mathbb{I}_{(d-k)\times (d-k)}
      \end{array}\right),
\]
where $\mathbb{J}_{k\times k}$ and $\mathbb{I}_{(d-k)\times (n-k)}$ are $k\times k$ and $(d-k)\times (d-k)$ operator matrices respectively and are defined by
$$
\mathbb{J}_{k\times k}=\begin{pmatrix}
0       &\cdots  &0 &I\\
\vdots &  \iddots     &I  &0\\
0  &I\smash{\makebox[0pt][l]{\;\raisebox{0.8em}{$\iddots$}}}  &  \iddots      &\vdots\\
I &0  &\cdots &0
\end{pmatrix}\;\text{ and }\;\mathbb{I}_{(d-k)\times (d-k)}=\begin{pmatrix}
I       &0 &\cdots  &0\\
0 & I \smash{\makebox[0pt][l]{\;\raisebox{-0.8em}{$\ddots$}}}     &\ddots  &\vdots\\
\vdots  &\ddots  &I       &0\\
0 &\cdots  &0 &I
\end{pmatrix}
$$
In view of Lemma \ref{ir2020}, we have $\mathbb{U}_k\in \mathcal{B}_{\mathbb{A}}(\mathcal{H}\oplus \mathcal{H})$ for all $k$. Moreover, a short calculation shows that $\mathbb{U}_k^{\sharp_{\mathbb{A}}}=\mathbb{P}\mathbb{U}_k$ where $\mathbb{P}=\begin{pmatrix}
P_{\overline{\mathcal{R}(A)}}       &0 &\ldots  &0\\
0 & P_{\overline{\mathcal{R}(A)}}     &\ddots  &\vdots\\
\vdots  &\ddots  &P_{\overline{\mathcal{R}(A)}}      &0\\
0 &\ldots  &0 &P_{\overline{\mathcal{R}(A)}}
\end{pmatrix}$. So, it is not difficult to verify that $\mathbb{U}_k$ is $\mathbb{A}$-unitary operator for all $k$.  Moreover, one can check that
\begin{align*}
\omega_{\mathbb{A}}(\mathbb{T})
& \leq\omega_{\mathbb{A}}\left[\begin{pmatrix}
T_{11} & T_{12} &\cdots& T_{1d}\\
0 &0 &\cdots& 0\\
\vdots & \vdots & \vdots & \vdots\\
0 & 0 &\cdots& 0\\
\end{pmatrix}\right]+\omega_{\mathbb{A}}\left[\begin{pmatrix}
0 &0 &\cdots& 0\\
T_{21} & T_{22} &\cdots& T_{2d}\\
0 &0 &\cdots& 0\\
\vdots & \vdots & \vdots & \vdots\\
0 & 0 &\cdots& 0\\
\end{pmatrix}\right] \\
&\;\;\;+\ldots\ldots+\omega_{\mathbb{A}}\left[\begin{pmatrix}
0 &0 &\cdots& 0\\
\vdots & \vdots & \vdots & \vdots\\
0 & 0 &\cdots& 0\\
T_{d1} & T_{d2} &\cdots& T_{dd}\\
\end{pmatrix}\right] \\
&=\omega_{\mathbb{A}}\left[\begin{pmatrix}
T_{11} & T_{12} &\cdots& T_{1d}\\
0 &0 &\cdots& 0\\
\vdots & \vdots & \vdots & \vdots\\
0 & 0 &\cdots& 0\\
\end{pmatrix}\right]+\omega_{\mathbb{A}}\left[\mathbb{U}_2^{\sharp_{\mathbb{A}}}\begin{pmatrix}
T_{22} & T_{21} &\cdots& T_{2d}\\
0 &0 &\cdots& 0\\
\vdots & \vdots & \vdots & \vdots\\
0 & 0 &\cdots& 0\\
\end{pmatrix}\mathbb{U}\right] \\
&\;\;\;+\ldots\ldots+\omega_{\mathbb{A}}\left[\mathbb{U}_d^{\sharp_{\mathbb{A}}}\begin{pmatrix}
T_{dd} & T_{dd-1} &\cdots& T_{d1}\\
0 &0 &\cdots& 0\\
\vdots & \vdots & \vdots & \vdots\\
0 & 0 &\cdots& 0\\
\end{pmatrix}\mathbb{U}\right] \\
\end{align*}
So, by using Lemma \ref{weak} together with \eqref{first01}, we obtain
\begin{align*}
\omega_{\mathbb{A}}(\mathbb{T})
& \leq\omega_{\mathbb{A}}\left[\begin{pmatrix}
T_{11} & T_{12} &\cdots& T_{1d}\\
0 &0 &\cdots& 0\\
\vdots & \vdots & \vdots & \vdots\\
0 & 0 &\cdots& 0\\
\end{pmatrix}\right]+\omega_{\mathbb{A}}\left[\begin{pmatrix}
T_{22} & T_{21} &\cdots& T_{2d}\\
0 &0 &\cdots& 0\\
\vdots & \vdots & \vdots & \vdots\\
0 & 0 &\cdots& 0\\
\end{pmatrix}\right] \\
&\;\;\;+\ldots\ldots+\omega_{\mathbb{A}}\left[\begin{pmatrix}
T_{dd} & T_{dd-1} &\cdots& T_{d1}\\
0 &0 &\cdots& 0\\
\vdots & \vdots & \vdots & \vdots\\
0 & 0 &\cdots& 0\\
\end{pmatrix}\right] \\
&\leq\frac{1}{2}\left(\|T_{11}\|_A+\sqrt{\left\|\sum^d_{j=1}T_{1j}T_{1j}^{\sharp_A}\right\|_A}\right)+\frac{1}{2}\left(\|T_{22}\|_A+\sqrt{\left\|\sum^d_{j=1,j\neq 2}T_{2j}T_{2j}^{\sharp_A}\right\|_A}\right)\\
&\;\;\;+\ldots\ldots+\frac{1}{2}\left(\|T_{dd}\|_A+\sqrt{\left\|\sum^{d-1}_{j=1}T_{dj}T_{dj}^{\sharp_A}\right\|_A}\right).
\end{align*}
This finishes the proof of the theorem.
\end{proof}

To establish our next result, we shall require the following lemma.
\begin{lemma}\label{maxma}
Let $\mathbb{T}=\begin{pmatrix}
T_1 & 0 & 0\\
0 & \ddots & 0 \\
0 & 0 & T_d
\end{pmatrix}$ and $\mathbb{S}=\left( {\begin{array}{*{20}{c}}
   {0} & {} & {} & {{T_1}}  \\
    & {} & {{T_2}} & {}  \\
   {} &  {\mathinner{\mkern2mu\raise1pt\hbox{.}\mkern2mu
 \raise4pt\hbox{.}\mkern2mu\raise7pt\hbox{.}\mkern1mu}}  & {} & {}  \\
   {{T_d}} & {} &  & {0}  \\
\end{array}} \right)$be such that $T_i\in \mathcal{B}_{A^{1/2}}(\mathcal{H})$ for all $i\in\{1,\cdots,d\}$. Then, the following assertions hold
\begin{itemize}
  \item [(a)] $\|\mathbb{T}\|_\mathbb{A}=\max_{i\in\{1,\cdots,d\}}\|T_i\|_A$.
  \item [(b)] $\omega_{\mathbb{A}}(\mathbb{T})=\max_{i\in\{1,\cdots,d\}}\omega_{A}(T_i)$.
    \item [(c)] $\|\mathbb{S}\|_\mathbb{A}=\max_{i\in\{1,\cdots,d\}}\|T_i\|_A$.
\end{itemize}
\end{lemma}
\begin{proof}
\noindent (a)\;Let $x=(x_1,\cdots,x_d),y=(y_1,\cdots,y_d)\in \mathbb{H}$. By using the Cauchy-Schwarz inequality and the arithmetic-geometric mean inequality, we obtain
\begin{align}\label{k0}
	|\langle \mathbb{T}x,y \rangle_\mathbb{A}|\nonumber
&\leq \sum_{k=1}^d|\langle T_kx_k\mid y_k \rangle_A|\nonumber \\
	&\leq \sum_{k=1}^d\|T_k\|_A\|x_k\|_A\|y_k\|_A \nonumber\\
	&\leq\left(\max_{i\in\{1,\cdots,d\}}\|T_i\|_A\right)\times\tfrac{1}{2}\sum_{k=1}^d\left(\|x_k\|_A^2+\|y_k\|_A^2\right)\nonumber\\
&=\frac{\|x\|_\mathbb{A}^2+\|y\|_\mathbb{A}^2}{2}\left(\max_{i\in\{1,\cdots,d\}}\|T_i\|_A\right).
\end{align}
By taking the supremum over all $x\in \mathbb{H}$ with $\|x\|_\mathbb{A} =1$ in the inequality \eqref{k0} and then using \eqref{fg}, we get
$$\|\mathbb{T}\|_\mathbb{A}\leq\max_{i\in\{1,\cdots,d\}}\|T_i\|_A.$$
Let $u=(x,0,\cdots,0)\in \mathbb{H}$ and $v=(y,0,\cdots,0)\in \mathbb{H}$ be such that $\|x\|_A=\|y\|_A=1.$ Then $\|u\|_\mathbb{A}=\|v\|_\mathbb{A}=\|x\|_A=\|y\|_A=1.$ Therefore, in view of \eqref{fg}, we have
\begin{align*}
\|\mathbb{T}\|_\mathbb{A}
&\geq |\langle \mathbb{T}u,v\rangle_\mathbb{A}|= |\langle T_1x\mid y\rangle_A|.
\end{align*}
This implies that $\|\mathbb{T}\|_\mathbb{A}\geq \|T_1\|_A$. Similarly, we can show that $\|\mathbb{T}\|_\mathbb{A}\geq \|T_k\|_A$ for all $k\in\{2,3,\ldots,d\}$. This proves the desired equality.
\par \vskip 0.1 cm \noindent (b)\;Follows by using similar arguments as in $(a)$.
\par \vskip 0.1 cm \noindent (c)\;$x=(x_1,\cdots,x_d)\in \mathbb{H}$. By using \eqref{semiiineq}, it can be observed that
\begin{align*}
\left\|\mathbb{S}x\right\|_{\mathbb{A}}^2
 &=\|T_1x_d\|_A^2+\|T_2x_{d-1}\|_A^2+\cdots+\|T_dx_1\|_A^2\\
  &\leq \left(\max_{i\in\{1,\cdots,d\}}\|T_i\|_A^2\right) \sum_{k=1}^d\|x_k\|_A^2.
\end{align*}
This yields that
$$\left\|\mathbb{S}\right\|_{\mathbb{A}}\leq \max_{i\in\{1,\cdots,d\}}\|T_i\|_A.$$
Let $x_1\in  \mathcal{H}$ be such that $\|x\|_A =1$ and $u=(0,0,\cdots,x_1)$. Clearly, $\|u\|_\mathbb{A} =1.$ So, we obtain
$$\left\|\mathbb{S}\right\|_{\mathbb{A}}\geq\left\|\mathbb{S}u\right\|_{\mathbb{A}}=\|T_1x_1\|_A.$$
Thus, by taking the supremum over all $x_1\in \mathcal{H}$ with $\|x_1\|_A=1$, we obtain $\left\|\mathbb{S}\right\|_{\mathbb{A}}\geq\|T_1\|_A.$ Similarly, it is not difficult to prove that $\|\mathbb{S}\|_\mathbb{A}\geq \|T_i\|_A$ for all $i\in\{2,3,\ldots,d\}$. This proves the desired equality.
\end{proof}
Our next result is stated as follows.
\begin{theorem}
Let $\mathbb{T}=(T_{ij})$ be a ${d\times d}$ operator matrix where $T_{ij}\in \mathcal{B}_{A}(\mathcal{H})$. Then,
\begin{equation}\label{r2}
\omega_{\mathbb{A}}(\mathbb{T})\leq \frac{1}{2}\sum_{i=1}^d\omega_A(T_{ii})+\frac{1}{4}\left(d+\sum^d_{i,j=1}\|T_{ij}\|_A^2\right).
\end{equation}
\end{theorem}
\begin{proof}
We first prove that
\begin{equation}\label{first02}
\omega_{\mathbb{A}}(\mathbb{S})\leq \frac{\omega_A(T_{11})}{2}+\frac{1}{4}+\frac{1}{4}\left\|\sum^d_{j=1}T_{1j}T_{1j}^{\sharp_A}\right\|_A,
\end{equation}
where $\mathbb{S}=\begin{pmatrix}
T_{11} & T_{12} &\cdots& T_{1d}\\
0 &0 &\cdots& 0\\
\vdots & \vdots & \vdots & \vdots\\
0 & 0 &\cdots& 0\\
\end{pmatrix}$. Let $\theta\in \mathbb{R}$. By proceeding as in the proof of Theorem \ref{thf1} we get
\begin{align*}
\|\Re_A(e^{i\theta}\mathbb{S})\|_{\mathbb{A}}
 &=\frac{1}{2}r_{\mathbb{A}}\left[\begin{pmatrix}
e^{i\theta}T_{11}+e^{-i\theta}T_{11}^{\sharp_{A}} & e^{i\theta}T_{12} &\cdots& e^{i\theta}T_{1d}\\
e^{-i\theta}T_{12}^{\sharp_{A}} &0 &\cdots& 0\\
\vdots & \vdots & \vdots & \vdots\\
e^{-i\theta}T_{1d}^{\sharp_{A}} & 0 &\cdots& 0\\
\end{pmatrix}\right]\\
&=\frac{1}{2}r_{\mathbb{A}}\left[
\begin{pmatrix}
T_{11}^{\sharp_{A}} & 0 &\cdots& e^{i\theta}I\\
T_{12}^{\sharp_{A}} &0 &\cdots& 0\\
\vdots & \vdots & \vdots & \vdots\\
T_{1d}^{\sharp_{A}} & 0 &\cdots& 0\\
\end{pmatrix}
\begin{pmatrix}
e^{-i\theta}I &0 &\cdots& 0\\
0 & 0 &\cdots& 0\\
\vdots & \vdots & \vdots & \vdots\\
T_{11} & T_{12} &\cdots& T_{1d}\\
\end{pmatrix}
\right].
\end{align*}
So, by using \eqref{commut} and \eqref{dom} we get
\begin{align*}
\|\Re_A(e^{i\theta}\mathbb{S})\|_{\mathbb{A}}
&=\frac{1}{2}r_{\mathbb{A}}\left[
\begin{pmatrix}
e^{-i\theta}I &0 &\cdots& 0\\
0 & 0 &\cdots& 0\\
\vdots & \vdots & \vdots & \vdots\\
T_{11} & T_{12} &\cdots& T_{1d}\\
\end{pmatrix}
\begin{pmatrix}
T_{11}^{\sharp_{A}} & 0 &\cdots& e^{i\theta}I\\
T_{12}^{\sharp_{A}} &0 &\cdots& 0\\
\vdots & \vdots & \vdots & \vdots\\
T_{1d}^{\sharp_{A}} & 0 &\cdots& 0\\
\end{pmatrix}
\right]\\
&=\tfrac{1}{2}r_{\mathbb{A}}\left[
\begin{pmatrix}
e^{-i\theta}T_{11}^{\sharp_{A}} & 0&\cdots&0& I\\
0 & 0 &\cdots&0& 0\\
\vdots & \vdots & \vdots & \vdots& \vdots\\
\sum_{k=1}^d T_{1k}T_{1k}^{\sharp_{A}}& 0 &\cdots&0& e^{i\theta}T_{11}\\
\end{pmatrix}
\right]\\
&\leq\tfrac{1}{2}\omega_{\mathbb{A}}\left[
\begin{pmatrix}
e^{-i\theta}T_{11}^{\sharp_{A}} & 0&\cdots&0& I\\
0 & 0 &\cdots&0& 0\\
\vdots & \vdots & \vdots & \vdots& \vdots\\
\sum_{k=1}^d T_{1k}T_{1k}^{\sharp_{A}}& 0 &\cdots&0& e^{i\theta}T_{11}\\
\end{pmatrix}
\right]\\
&\leq\tfrac{1}{2}\omega_{\mathbb{A}}\left[
\begin{pmatrix}
e^{-i\theta}T_{11}^{\sharp_{A}} & 0&\cdots&0& 0\\
0 & 0 &\cdots&0& 0\\
\vdots & \vdots & \vdots & \vdots& \vdots\\
0& 0 &\cdots&0& e^{i\theta}T_{11}\\
\end{pmatrix}
\right]+\tfrac{1}{2}\omega_{\mathbb{A}}\left[
\begin{pmatrix}
 0&\cdots&0& I\\
 0 &\cdots&0& 0\\
 \vdots & \vdots & \vdots& \vdots\\
 0 &\cdots&0& 0\\
\end{pmatrix}
\right]\\
&+   \tfrac{1}{2}\omega_{\mathbb{A}}\left[
\begin{pmatrix}
0 & 0&\cdots&0\\
0 & 0 &\cdots&0\\
\vdots & \vdots & \vdots & \vdots\\
\sum_{k=1}^d T_{1k}T_{1k}^{\sharp_{A}}& 0 &\cdots&0\\
\end{pmatrix}
\right].
\end{align*}
So, by using Lemma \ref{maxma} together with \eqref{at2} we obtain
\begin{equation}
\|\Re_A(e^{i\theta}\mathbb{S})\|_{\mathbb{A}}\leq \frac{\omega_A(T_{11})}{2}+\frac{1}{4}+\frac{1}{4}\left\|\sum^d_{j=1}T_{1j}T_{1j}^{\sharp_A}\right\|_A,
\end{equation}
So, by taking the supremum over all $\theta\in \mathbb{R}$ in the above inequality and then using \eqref{zamnum} we get \eqref{first02}. Finally, by using an argument similar to that used in proof of Theorem \ref{thf1} we reach the desired equality \eqref{r2}.
\end{proof}

In order to prove our next result in this section, we need the following lemmas.
\begin{lemma}(\cite{feki03})\label{lemmajdid}
Let $\mathbb{T}= (T_{ij})_{d \times d}$ be such that $T_{ij}\in \mathcal{B}_{A^{1/2}}(\mathcal{H})$ for all $i,j$. Then, $\mathbb{T}\in\mathcal{B}_{\mathbb{A}^{1/2}}(\mathbb{H})$. Moreover, we have
\begin{equation}\label{tag0}
\|\mathbb{T}\|_{\mathbb{A}}\leq \|\widehat{\mathbb{T}}^{\mathbb{A}}\|,
\end{equation}
where $\widehat{\mathbb{T}}^{\mathbb{A}}=(\|T_{ij} \|_A)_{d\times d}\in \mathbb{M}_d(\mathbb{C})$.
\end{lemma}

\begin{lemma}\label{jdidddd}
Let $T,S\in \mathcal{B}_{A}(\mathcal{H})$ and $\mathbb{B}=\begin{pmatrix}
A &0\\
0 &A
\end{pmatrix}$. Then,
\begin{equation*}
\omega_{\mathbb{B}}\left[\begin{pmatrix}
0&T\\
S &0
\end{pmatrix}\right]=\frac{1}{2}\sup_{\theta\in \mathbb{R}}\left\|e^{i\theta}T+e^{-i\theta}S^{\sharp_A}\right\|_A.
\end{equation*}
\end{lemma}
\begin{proof}
Notice first that, in view of Lemma \ref{maxma} (c) we have
\begin{equation}\label{particular}
\left\|\begin{pmatrix}
0&X\\
X^{\sharp_A} &0
\end{pmatrix}\right\|_{\mathbb{B}}=\|X\|_A,
\end{equation}
for every $X\in \mathcal{B}_{A}(\mathcal{H})$. Now, by using Lemma \ref{ir2020} and \eqref{zamnum}, it follows that
\begin{align*}
\omega_{\mathbb{B}}\left[\begin{pmatrix}
0&T\\
S &0
\end{pmatrix}\right]
& =\omega_{\mathbb{B}}\left[\begin{pmatrix}
0&T\\
S &0
\end{pmatrix}^{\sharp_\mathbb{B}}\right]\\
 &=\frac{1}{2}\sup_{\theta\in \mathbb{R}}\left\|e^{i\theta}\begin{pmatrix}
0&S^{\sharp_A}\\
T^{\sharp_A} &0
\end{pmatrix}+e^{-i\theta}\begin{pmatrix}
0& (T^{\sharp_A})^{\sharp_A}\\
(S^{\sharp_A})^{\sharp_A} &0
\end{pmatrix}\right\|_\mathbb{B}\\
  &=\frac{1}{2}\sup_{\theta\in \mathbb{R}}\left\|\begin{pmatrix}
0&e^{i\theta}S^{\sharp_A}+e^{-i\theta}(T^{\sharp_A})^{\sharp_A}\\
e^{i\theta}T^{\sharp_A}+e^{-i\theta}(S^{\sharp_A})^{\sharp_A} &0
\end{pmatrix}\right\|_\mathbb{B}\\
  &=\frac{1}{2}\sup_{\theta\in \mathbb{R}}\left\|\begin{pmatrix}
0&e^{i\theta}S^{\sharp_A}+e^{-i\theta}(T^{\sharp_A})^{\sharp_A}\\
[e^{i\theta}S^{\sharp_A}+e^{-i\theta}(T^{\sharp_A})^{\sharp_A}]^{\sharp_A} &0
\end{pmatrix}\right\|_\mathbb{B}\\
  &=\frac{1}{2}\sup_{\theta\in \mathbb{R}}\left\|e^{i\theta}S^{\sharp_A}+e^{-i\theta}(T^{\sharp_A})^{\sharp_A}\right\|_A,\;(\text{by }\eqref{particular})\\
 &=\frac{1}{2}\sup_{\theta\in \mathbb{R}}\left\|e^{i\theta}T^{\sharp_A}+e^{-i\theta}S\right\|_A.
\end{align*}
So, by replacing $\theta$ by $-\theta$ in the above equality, we obtain
\begin{equation}\label{durr2}
\omega_{\mathbb{B}}\left[\begin{pmatrix}
0&T\\
S &0
\end{pmatrix}\right]=\frac{1}{2}\sup_{\theta\in \mathbb{R}}\left\|e^{-i\theta}T^{\sharp_A}+e^{i\theta}S\right\|_A.
\end{equation}
Let $\mathbb{U}=\begin{pmatrix}
0&I\\
I&0
\end{pmatrix}.$ In view of Lemma \ref{ir2020}, we have $\mathbb{U}\in \mathcal{B}_{\mathbb{B}}(\mathcal{H}\oplus \mathcal{H})$ and $\mathbb{U}^{\sharp_{\mathbb{B}}}=\begin{pmatrix}
0&P_{\overline{\mathcal{R}(A)}} \\
P_{\overline{\mathcal{R}(A)}}&0
\end{pmatrix}.$ So, we verify that $\|\mathbb{U}x\|_\mathbb{B}=\|\mathbb{U}^{\sharp_\mathbb{B}}x\|_\mathbb{B}=\|x\|_\mathbb{B}$ for all $x=(x_1,x_2)\in \mathcal{H}\oplus \mathcal{H}$. Hence, $\mathbb{U}$ is $\mathbb{B}$-unitary operator. Thus, by Lemma \ref{weak} we have
\begin{align}\label{offfd}
\omega_{\mathbb{B}}\left[\begin{pmatrix}
0&T\\
S &0
\end{pmatrix}\right]
& =\omega_\mathbb{B}\left[\mathbb{U}^{\sharp_A}\begin{pmatrix}
0&T\\
S &0
\end{pmatrix}\mathbb{U}\right]\nonumber\\
& =\omega_\mathbb{B}\left[\begin{pmatrix}
P_{\overline{\mathcal{R}(A)}}&0\\
0 &P_{\overline{\mathcal{R}(A)}}
\end{pmatrix}\begin{pmatrix}
0&S\\
T &0
\end{pmatrix}\right]\nonumber\\
&=\omega_{\mathbb{B}}\left[\begin{pmatrix}
0&S\\
T &0
\end{pmatrix}\right]
\end{align}
Hence, by combining \eqref{durr2} together with \eqref{offfd} we prove the desired equality.
\end{proof}

Now, we are in a position to establish the following result which generalizes \cite[Theorem 4.17.]{BP}.
\begin{theorem}\label{th-2}
Let $\mathbb{T}=(T_{ij})$ be an $d\times d$ operator matrix with $T_{ij}\in \mathcal{B}_A(\mathcal{H})$. Then,
\[\omega_\mathbb{A}(\mathbb{T})\leq \omega(S),\]
where $S=[s_{ij}]\in \mathbb{M}_d(\mathbb{C})$ is given by
\begin{equation}\label{sij}
s_{ij}=
\begin{cases}
\omega(T_{ij})&,\text{ if }\;i=j,\\
\omega_\mathbb{B}\left[\begin{pmatrix}
0 &T_{ij}\\
T_{ji}&0
\end{pmatrix}\right]\text{ with } \mathbb{B}=\begin{pmatrix}
A&0\\
0&A
\end{pmatrix}&,\text{ if }\;i\neq j.
\end{cases}.
\end{equation}
\end{theorem}
\begin{proof}
Let $\theta\in \mathbb{R}$ and $\mathbb{B}=\begin{pmatrix}
A&0\\
0&A
\end{pmatrix}$. By using \eqref{ir2020}, it can be seen that
\begin{align*}
\Re_\mathbb{A}(e^{i\theta}\mathbb{T})
& = \frac{1}{2}\left(e^{i\theta}\mathbb{T}+e^{-i\theta}\mathbb{T}^{\sharp_\mathbb{A}}\right)\\
 &=\begin{pmatrix}
\Re_A(e^{i\theta}T_{11}) &\tfrac{1}{2}(e^{i\theta}T_{12}+e^{-i\theta}T_{21}^{\sharp_A})  &\cdots&\tfrac{1}{2}(e^{i\theta}T_{1d}+e^{-i\theta}T_{d1}^{\sharp_A}) \\
\tfrac{1}{2}(e^{i\theta}T_{21}+e^{-i\theta}T_{12}^{\sharp_A})  &   \Re_A(e^{i\theta}T_{22}) &\cdots& \tfrac{1}{2}(e^{i\theta}T_{2d}+e^{-i\theta}T_{d2}^{\sharp_A}) \\
\vdots & \vdots & \vdots & \vdots\\
\tfrac{1}{2}(e^{i\theta}T_{d1}+e^{-i\theta}T_{1d}^{\sharp_A})   &   \tfrac{1}{2}(e^{i\theta}T_{d2}+e^{-i\theta}T_{2d}^{\sharp_A})  &\cdots& \Re_A(e^{i\theta}T_{dd})\\
\end{pmatrix}.
\end{align*}
So, by applying Lemma \eqref{lemmajdid} together with the norm monotonicity of matrices with nonnegative entries and then using Lemma \ref{jdidddd} and \eqref{zamnum} we get
\begin{align*}
&\left\|\Re_\mathbb{A}(e^{i\theta}\mathbb{T})\right\|_\mathbb{A}\\
 &\leq\left\|\begin{pmatrix}
\omega_A(T_{11}) &\omega_\mathbb{B}\left[\begin{pmatrix}
0 &T_{12}\\
T_{21}&0
\end{pmatrix}\right]  &\cdots&\omega_\mathbb{B}\left[\begin{pmatrix}
0 &T_{1d}\\
T_{d1}&0
\end{pmatrix}\right] \\
\omega_\mathbb{B}\left[\begin{pmatrix}
0 &T_{21}\\
T_{12}&0
\end{pmatrix}\right]  &   \omega_A(T_{22}) &\cdots& \omega_\mathbb{B}\left[\begin{pmatrix}
0 &T_{2d}\\
T_{d2}&0
\end{pmatrix}\right] \\
\vdots & \vdots & \vdots & \vdots\\
\omega_\mathbb{B}\left[\begin{pmatrix}
0 &T_{d1}\\
T_{1d}&0
\end{pmatrix}\right]   &   \omega_\mathbb{B}\left[\begin{pmatrix}
0 &T_{d2}\\
T_{2d}&0
\end{pmatrix}\right]  &\cdots& \omega_A(T_{dd})\\
\end{pmatrix}\right\|.
\end{align*}
So, by taking the supremum over all $\theta\in \mathbb{R}$ in the above inequality we get $\omega_\mathbb{A}(\mathbb{T})\leq \|S\|$ where $S$ is defined in \eqref{sij}. Finally, by \eqref{offfd}, we have $\omega_\mathbb{B}\left[\begin{pmatrix}
0 &T_{ij}\\
T_{ji}&0
\end{pmatrix}\right]=\omega_\mathbb{B}\left[\begin{pmatrix}
0 &T_{ji}\\
T_{ij}&0
\end{pmatrix}\right]$ for all $i,j$. Thus, $S$ is a real symmetric matrix and so $\omega(S)=\|S\|$. Therefore, we get the desired result.
\end{proof}

Next we state from \cite[p. 44]{HJ} the following useful lemma.
\begin{lemma}\label{l-10}
Let $T=(t_{ij})\in \mathbb{M}_d(\mathbb{C})$ such that $t_{ij}\geq 0$ for all $i,j=1,2,\ldots,d.$ Then
$$\omega(T)=\frac{r\left (t_{ij}+t_{ji}\right)}{2}.$$
\end{lemma}
\begin{remark}
Bhunia et al. proved recently in \cite[Theorem 4.12.]{BP} that for a $d \times d$ operator matrix $\mathbb{T}=(T_{ij})$ with $T_{ij}\in \mathcal{B}_A(\mathcal{H})$. Then,
\begin{equation}\label{pint2020}
\omega_\mathbb{A}(\mathbb{T})\leq \omega([t_{ij}])\;\;\text{ where }\;\; t_{ij}= \begin{cases}
\omega_{A}(T_{ij}), &  i=j \\
\|T_{ij}\|_A, &  i\neq j.
\end{cases}
\end{equation}
Clearly, by Lemma \ref{jdidddd} one observes that
\begin{equation*}
\omega_{\mathbb{B}}\left[\begin{pmatrix}
0&T\\
S &0
\end{pmatrix}\right]\leq\frac{\|T\|_A+\|S\|_A}{2},
\end{equation*}
for every $T,S\in \mathcal{B}_{A}(\mathcal{H})$. So, by taking into consideration Lemma \ref{l-10}, it is not difficult to verify that the inequality proved in Theorem \ref{th-2} refines the inequality \eqref{pint2020}.
\end{remark}

Our next result is stated as follows.
\begin{theorem}
Let $\mathbb{T}=(T_{ij})$ be a ${d\times d}$ operator matrix where $T_{ij}\in \mathcal{B}_{A}(\mathcal{H})$. Then,
\begin{equation*}
\omega_{\mathbb{A}}(\mathbb{T})\leq \frac{1}{2}\sum_{i=1}^d\left(\omega_A(T_{ii})+\sqrt{\omega_A^2(T_{ii})+\sum^d_{j=1,j\neq i}\|T_{ij}\|_A^2}\right).
\end{equation*}
\end{theorem}
\begin{proof}
We first prove that
\begin{equation}\label{first004}
\omega_{\mathbb{A}}\left[\begin{pmatrix}
T_{11} & T_{12} &\cdots& T_{1d}\\
0 &0 &\cdots& 0\\
\vdots & \vdots & \vdots & \vdots\\
0 & 0 &\cdots& 0\\
\end{pmatrix}\right]\leq \frac{1}{2}\left(\omega_A(T_{11})+\sqrt{\omega_A^2(T_{11}) +\sum^d_{j=2}\|T_{1j}\|_A^2 }\right).
\end{equation}
By applying Theorem \ref{th-2} we obtain
\begin{align*}
&\omega_{\mathbb{A}}\left[\begin{pmatrix}
T_{11} & T_{12} &\cdots& T_{1d}\\
0 &0 &\cdots& 0\\
\vdots & \vdots & \vdots & \vdots\\
0 & 0 &\cdots& 0\\
\end{pmatrix}\right]\\
&\leq\omega\left[\begin{pmatrix}
\omega_A(T_{11}) &\omega_\mathbb{B}\left[\begin{pmatrix}
0 &T_{12}\\
0&0
\end{pmatrix}\right]  &\cdots&\omega_\mathbb{B}\left[\begin{pmatrix}
0 &T_{1d}\\
0&0
\end{pmatrix}\right] \\
\omega_\mathbb{B}\left[\begin{pmatrix}
0 &0\\
T_{12}&0
\end{pmatrix}\right]  &   0 &\cdots& 0 \\
\vdots & \vdots & \vdots & \vdots\\
\omega_\mathbb{B}\left[\begin{pmatrix}
0 &0\\
T_{1d}&0
\end{pmatrix}\right]   &   0  &\cdots& 0\\
\end{pmatrix}\right].
\end{align*}
Moreover, since $\mathbb{B}\begin{pmatrix}
0 &0\\
T_{1j}&0
\end{pmatrix}^2=\begin{pmatrix}
0 &0\\
0&0
\end{pmatrix}$ for every $j\in\{1,\cdots,d\}$, then by applying \eqref{at2} together with Lemma \ref{maxma} (c) we have
$$\omega_\mathbb{B}\left[\begin{pmatrix}
0 &0\\
T_{1j}&0
\end{pmatrix}\right]=\frac{1}{2}\left\|\begin{pmatrix}
0 &0\\
T_{1j}&0
\end{pmatrix}\right\|_\mathbb{B}=\frac{1}{2}\|T_{1j}\|_A.$$
So, we obtain
\begin{align*}
&\omega_{\mathbb{A}}\left[\begin{pmatrix}
T_{11} & T_{12} &\cdots& T_{1d}\\
0 &0 &\cdots& 0\\
\vdots & \vdots & \vdots & \vdots\\
0 & 0 &\cdots& 0\\
\end{pmatrix}\right]\\
&\leq\omega\left[\begin{pmatrix}
\omega_A(T_{11}) &\frac{\|T_{12}\|_A}{2}  &\cdots&\frac{\|T_{1d}\|_A}{2} \\
\frac{\|T_{12}\|_A}{2} &   0 &\cdots& 0 \\
\vdots & \vdots & \vdots & \vdots\\
\frac{\|T_{1d}\|_A}{2}   &   0  &\cdots& 0\\
\end{pmatrix}\right]\\
 &=\frac{1}{2}r\left[\begin{pmatrix}
2\omega_A(T_{11}) &\|T_{12}\|_A  &\cdots&\|T_{1d}\|_A \\
\|T_{12}\|_A &   0 &\cdots& 0 \\
\vdots & \vdots & \vdots & \vdots\\
\|T_{1d}\|_A   &   0  &\cdots& 0\\
\end{pmatrix}\right]\;(\text{ by Lemma }\ref{l-10})\\
&=\frac{1}{2}\left(\omega_A(T_{11})+\sqrt{\omega_A^2(T_{11}) +\sum^d_{j=2}\|T_{1j}\|_A^2 }\right).
\end{align*}
This proves \eqref{first004}. Now, by proceeding as in proof of Theorem \ref{thf1} we get the required result.
\end{proof}

The following lemma is useful in proving our next result.
\begin{lemma}\label{lemma:2}
Let $T \in \mathcal{B}_A(\mathcal{H})$. Then
\begin{equation*}
\omega_A(T) \leq\sqrt{ \|\Re_A(T)\|_A^2+\|\Im_A(T)\|_A^2}.
\end{equation*}
\end{lemma}
\begin{proof}
Let $x\in \mathcal{H}$. Since $\Re_A(T)$ and $\Im_A(T)$ are $A$-selfadjoint operators, then by taking into consideration \eqref{aself1} we see that
\begin{align*}
\big|{\langle Tx\mid x \rangle}_A\big|^2
&= \big|{\langle \Re_A(T)x\mid x \rangle}_A+i{\langle \Im_A(T)x\mid x \rangle}_A\big|^2\\
& = \big|{\langle\Re_A(T)x\mid x\rangle}_A\big|^2 + \big|{\langle \Im_A(T)x\mid x\rangle}_A\big|^2\\
&\leq \|\Re_A(T)\|_A^2+\|\Im_A(T)\|_A^2.
\end{align*}
So, by taking the supremum over all $x\in \mathcal{H}$ with $\|x\|_A=1$ we get required result.
\end{proof}

\begin{theorem} \label{theorem:upper bound oprt 2}
Let $\mathbb{T}=(T_{ij})$ be a ${d\times d}$ operator matrix where $T_{ij}\in \mathcal{B}_{A}(\mathcal{H})$. Then,
\begin{equation*}
\omega_A(\mathbb{T})\leq  \frac{1}{2}\sum^d_{i=1}\sqrt{\lambda^2_i+\mu^2_i},
\end{equation*}
where
\begin{align*}
\lambda_i &=\|\Re_A(T_{ii})\|_A+\sqrt{\|\Re_A(T_{ii})\|_A^2+\sum_{j=1,j\neq i}^d\|T_{ij}\|_A^2} \;\;\text{ and }\\
\mu_i &=\|\Im_A(T_{ii})\|_A+\sqrt{\|\Im_A(T_{ii})\|_A^2+\sum_{j=1,j\neq i}^d\|T_{ij}\|_A^2}.
\end{align*}
\end{theorem}
\begin{proof}
Let $\mathbb{S}=\left(\begin{array}{cccc}
   T_{11}&T_{12}&\ldots &T_{1d} \\
    0&0&\ldots &0\\
    \vdots & \vdots & &\vdots  \\
     0&0&\ldots&0
    \end{array}\right).$ It is not difficult to verify that
\begin{align*}
\left\|\Re_\mathbb{A}(\mathbb{T})\right\|_\mathbb{A}
&=r_\mathbb{A}\left[\Re_\mathbb{A}(\mathbb{T})\right] \\
&=r_\mathbb{A}\left[\left(\begin{array}{cccc}
    \Re_A(T_{11})&\frac{T_{12}}{2}&\ldots&\frac{T_{1d}}{2} \\
    \frac{T_{12}^{\sharp_A}}{2}&0&\dots &0\\
	  \vdots& \vdots& &\vdots \\
    \frac{T_{1d}^{\sharp_A}}{2}&0&\ldots&0
    \end{array}\right)\right] \\
 &\leq r\left[ \left(\begin{array}{cccc}
    \|\Re_A(T_{11})\|_A&\frac{\|T_{12}\|_A}{2}&\ldots&\frac{\|T_{1d}\|_A}{2} \\
    \frac{\|T_{12}^{\sharp_A}\|_A}{2}&0&\ldots&0\\
    \vdots& \vdots& &\vdots \\
    \frac{\|T_{1d}^{\sharp_A}\|_A}{2}&0&\ldots&0
    \end{array}\right) \right \|\;(\text{ by Lemma }\ref{mjom})\\
    &= \frac{1}{2}\left(\|\Re_A(T_{11})\|_A+\sqrt{\|\Re_A(T_{11})\|_A^2+\sum_{j=2}^d\|T_{1j}\|_A^2}\right).
\end{align*}
Now, it can be seen that
		$$\Im_\mathbb{A}(\mathbb{T})=\left(\begin{array}{cccc}
    \Im_A(T_{11})&\frac{T_{12}}{2i}&\ldots&\frac{T_{1d}}{2i} \\
    -\frac{T^{\sharp_A}_{12}}{2i}&0&\ldots&0\\
	   \vdots& \vdots& &\vdots \\
    -\frac{T^{\sharp_A}_{1d}}{2i}&0&\ldots&0
    \end{array}\right).$$
Similarly, we prove that
\begin{align*}	
\|\Im_\mathbb{A}(\mathbb{T})\|_\mathbb{A}\leq \frac{1}{2}\left(\|\Im_A(T_{11})\|_A+\sqrt{\|\Im_A(T_{11})\|_A^2+\sum_{j=2}^d\|T_{1j}\|_A^2}\right).
\end{align*}
Hence, by Lemma \ref{lemma:2}, we get
\begin{equation*}
\omega_{\mathbb{A}}\left(\mathbb{S}\right)\leq \frac{1}{2} \sqrt{\lambda^2+\mu^2},
\end{equation*}
where
\begin{align*}
\lambda&=\|\Re_A(T_{11})\|_A+\sqrt{\|\Re_A(T_{11})\|_A^2+\sum_{j=2}^d\|T_{1j}\|_A^2},\\
\mu&=\|\Im_A(T_{11})\|_A+\sqrt{\|\Im_A(T_{11})\|_A^2+\sum_{j=2}^d\|T_{1j}\|_A^2}.
\end{align*}
Finally, by using an argument similar to that used in proof of Theorem \ref{thf1} we reach the desired result.
\end{proof}

Our next result reads as follows.
\begin{theorem}
Let $\mathbb{T}=(T_{ij})$ be a ${d\times d}$ operator matrix where $T_{ij}\in \mathcal{B}_{A}(\mathcal{H})$. Then,
\begin{equation*}
\omega_{\mathbb{A}}(\mathbb{T})\leq \max_{i\in\{1,\cdots,d\}}\omega_A(T_{ii})+\frac{1}{2}\sum^d_{i=1}\sqrt{\left\|\sum^d_{j=1,j\neq i}T_{ij}T_{ij}^{\sharp_A}\right\|_A}.
\end{equation*}
\end{theorem}
\begin{proof}
By using the triangle inequality and Lemma \ref{maxma} (b) we get
\begin{align*}
\omega_{\mathbb{A}}(\mathbb{T})
& \leq\max_{i\in\{1,\cdots,d\}}\omega_A(T_{ii})+\omega_{\mathbb{A}}\left[\begin{pmatrix}
0 &T_{12} &\cdots& T_{1d}\\
0 & 0 &\cdots& 0\\
0 &0 &\cdots& 0\\
\vdots & \vdots & \vdots & \vdots\\
0 & 0 &\cdots& 0\\
\end{pmatrix}\right]
\\
&+\omega_{\mathbb{A}}\left[\begin{pmatrix}
0 &0 &\cdots& 0& 0\\
T_{21} & 0 &T_{23}&\cdots& T_{2d}\\
0 &0 &\cdots& 0& 0\\
\vdots & \vdots & \vdots & \vdots& \vdots\\
0 & 0 &\cdots& 0& 0\\
\end{pmatrix}\right]+\ldots+\omega_{\mathbb{A}}\left[\begin{pmatrix}
0 &0 &\cdots& 0& 0\\
\vdots & \vdots & \vdots & \vdots& \vdots\\
0 & 0 &\cdots& 0& 0\\
T_{d1} & T_{d2} &\cdots& T_{dd-1}& 0\\
\end{pmatrix}\right].
\end{align*}
On the other hand it can be seen that
\begin{align*}
&\mathbb{A}\begin{pmatrix}
0 &T_{12} &\cdots& T_{1d}\\
0 & 0 &\cdots& 0\\
0 &0 &\cdots& 0\\
\vdots & \vdots & \vdots & \vdots\\
0 & 0 &\cdots& 0\\
\end{pmatrix}^2
=\mathbb{A}\begin{pmatrix}
0 &0 &\cdots& 0& 0\\
T_{21} & 0 &T_{23}&\cdots& 0\\
0 &0 &\cdots& 0\\
\vdots & \vdots & \vdots & \vdots& \vdots\\
0 & 0 &\cdots& 0& 0\\
\end{pmatrix}^2 \\
 &=\ldots=\mathbb{A}\begin{pmatrix}
0 &0 &\cdots& 0\\
\vdots & \vdots & \vdots & \vdots\\
0 & 0 &\cdots& 0\\
T_{d1} & T_{d2} &\cdots& T_{dd}\\
\end{pmatrix}^2=\begin{pmatrix}
0 &0 &\cdots& 0\\
\vdots & \vdots & \vdots & \vdots\\
0 & 0 &\cdots& 0\\
0 & 0 &\cdots& 0\\
\end{pmatrix}.
\end{align*}
So, by \eqref{at2} we infer that
$$\mathbb{A}\begin{pmatrix}
0 &T_{12} &\cdots& T_{1d}\\
0 & 0 &\cdots& 0\\
0 &0 &\cdots& 0\\
\vdots & \vdots & \vdots & \vdots\\
0 & 0 &\cdots& 0\\
\end{pmatrix}=\frac{1}{2}\left\|\begin{pmatrix}
0 &T_{12} &\cdots& T_{1d}\\
0 & 0 &\cdots& 0\\
0 &0 &\cdots& 0\\
\vdots & \vdots & \vdots & \vdots\\
0 & 0 &\cdots& 0\\
\end{pmatrix}\right\|_\mathbb{A}.$$
Moreover, by using \eqref{diez} and Lemma \ref{maxma}, it can be checked that
\begin{align*}
\left\|\begin{pmatrix}
0 &T_{12} &\cdots& T_{1d}\\
0 & 0 &\cdots& 0\\
0 &0 &\cdots& 0\\
\vdots & \vdots & \vdots & \vdots\\
0 & 0 &\cdots& 0\\
\end{pmatrix}\right\|_\mathbb{A}^2
&=\left\|\begin{pmatrix}
0 &T_{12} &\cdots& T_{1d}\\
0 & 0 &\cdots& 0\\
0 &0 &\cdots& 0\\
\vdots & \vdots & \vdots & \vdots\\
0 & 0 &\cdots& 0\\
\end{pmatrix}\begin{pmatrix}
0 &T_{12} &\cdots& T_{1d}\\
0 & 0 &\cdots& 0\\
0 &0 &\cdots& 0\\
\vdots & \vdots & \vdots & \vdots\\
0 & 0 &\cdots& 0\\
\end{pmatrix}^{\sharp_{\mathbb{A}}}\right\|_\mathbb{A}\\
& =\left\|\begin{pmatrix}
\sum^d_{k=2}T_{1k}T_{1k}^{\sharp_A} &0 &\cdots& 0\\
0 & 0 &\cdots& 0\\
0 &0 &\cdots& 0\\
\vdots & \vdots & \vdots & \vdots\\
0 & 0 &\cdots& 0\\
\end{pmatrix}\right\|_\mathbb{A}=\left\|\sum^d_{k=2}T_{1k}T_{1k}^{\sharp_A}\right\|_A.
\end{align*}
Hence, by using similar arguments we get
\begin{align*}
\omega_{\mathbb{A}}(\mathbb{T})
& \leq\max_{i\in\{1,\cdots,d\}}\omega_A(T_{ii})+
\frac{1}{2}\sqrt{\left\|\sum^d_{k=2}T_{1k}T_{1k}^{\sharp_A}\right\|_A}
\\
&+\frac{1}{2}\sqrt{\left\|\sum^d_{k=1,k\neq2}T_{2k}T_{2k}^{\sharp_A}\right\|_A}+\ldots+
\frac{1}{2}\sqrt{\left\|\sum^d_{k=1,k\neq d}T_{dk}T_{dk}^{\sharp_A}\right\|_A}.
\end{align*}
This achieves the proof of the theorem.
\end{proof}

\end{document}